\documentclass[11pt]{amsart}

\usepackage{amssymb}
\usepackage{amsmath}
\usepackage[margin=1in]{geometry}
\usepackage[foot]{amsaddr}

\usepackage{epsfig}

\newtheorem{thm}{Theorem}[section]

\newtheorem{lem}[thm]{Lemma}

\theoremstyle{definition}
\newtheorem{definition}[thm]{Definition}

\theoremstyle{remark}

\numberwithin{equation}{section}

\newcommand{\N}{\mathbb{N}}

\begin{document}
\title{Star-coloring Splitting Graphs of Cycles}

\author{Sumun Iyer$^1$}
\address{$^1$Williams College}
\email{ssi1@williams.edu}

\begin{abstract}
A star coloring of a graph $G$ is a proper vertex coloring such that the subgraph induced by any pair of color classes is a star forest. The star chromatic number of $G$ is the minimum number of colors needed to star color $G$. In this paper we determine the star-chromatic number of the splitting graphs of cycles of length $n$ with $n \equiv 1 \pmod 3$ and $n=5$, resolving an open question of Furnma{\'n}czyk, Kowsalya, and Vernold Vivin.
\end{abstract}

\maketitle



\section{Introduction}
Let $G=(V,E)$ be a simple, undirected graph. A proper vertex $n$-coloring of $G$ is a surjective mapping $\phi: V \to \{1,2, \dots, n\}$ such that if $u$ is adjacent to $v$, then $\phi(u) \neq \phi(v)$. An \emph{n-star-coloring} of $G$ is a proper vertex $n$-coloring with one additional condition: each path on four vertices in $G$ is colored by at least three distinct colors. Alternatively, a star coloring of $G$ is a coloring such that the subgraph induced by any pair of color classes of $G$ is a star forest. Star coloring strengthens the notion of \emph{acyclic coloring} in which the subgraph induced by any pair of color classes is a forest. The \emph{star chromatic number} of $G$, denoted $\chi_s(G)$, is the minimum number of colors needed to star color $G$. 

Star coloring was first introduced by Branko Gr{\"u}nbaum in 1973 in the context of strengthening acyclic colorings of planar graphs \cite{grunbaum}. Star coloring also arises naturally in combinatorial computing. As one would imagine, finding an optimal star coloring of a general graph is NP-hard. Coleman and Mor{\'e} showed that star-coloring remains an NP-hard problem even on bipartite graphs \cite{coleman}. Coloring variants (like acyclic or star coloring) have been used to compute sparse Hessian and Jacobian matrices with techniques like finite differences and automatic differentiation. Gebremedhin, Tarafdar, Manne, and Pothen provided algorithms for finding heuristic solutions to star coloring and acyclic coloring problems \cite{gebre}. Their techniques utilize the structure of subgraphs induced by color classes and their findings have applications to efficient computation of Hessian matrices. Because the problems of computing these matrices can be recast as graph coloring problems, employing graph coloring as a model for computation can yield particularly effective algorithms. See \cite{colorsurvey} for a detailed survey of using graph coloring to compute derivatives.

In 2004 Fertin, Raspaud, and Reed determined the star chromatic number of trees, cycles, complete bipartite graphs, and other families of graphs \cite{fertin}. Star chromatic numbers of other types of graphs--including sparse graphs, bipartite planar graphs, and planar graphs with high girth--are studied in \cite{bu}, \cite{kierstead}, \cite{mohan}, and \cite{timmons}.

For a vertex $v$ of a graph $G = (V,E)$ let $N(v)$ denote the open neighborhood of $v$. The \emph{splitting graph}, $S(G)$, is obtained by adding a new vertex $v'$ corresponding to each $v$ in $V$ and edges from $v'$ such that $N(v)=N(v')$ (see Figure \ref{Fig1}). The splitting graph construction plays an important role in the theory of graph labeling. For a comprehensive survey of results in graph labeling and more references on splitting graphs see Gallian's ``Dynamic survey of graph labeling" \cite{gallian}. In 2017 Furma{\'n}czyk, Kowsalya, and Vernold Vivin determined the star chromatic number of splitting graphs of complete and complete bipartite graphs, paths, and some cycles \cite{furman}. They posed as an open question the problem of determining the  star chromatic number of splitting graphs of cycles on $n$ vertices where $n =5$ or $n \equiv 1 \pmod 3$. In this paper we provide a construction that shows that $\chi_S(S(C_n)) = 4$ for all $n \equiv 1 \pmod 3$, $n \geq 10$ and prove that $\chi_S(S(C_4))=\chi_S(S(C_5))=\chi_S(S(C_7)) = 5$.

\section{Splitting graphs of cycles}

We include the following two results for completeness. The first is due to Fertin, Raspaud, and Reed \cite{fertin} and the second is due to Furma{\'n}czyk, Kowsalya, and Vernold Vivin \cite{furman}.

\begin{thm}\label{cycle}\emph{(Fertin, Raspaud, Reed)}
Let $C_n$ be a cycle on $n \geq 3$ vertices. Then,
\[
\chi_S(C_n)=
\begin{cases}
4 & \textrm{when} \ n=5 \\
3 & \textrm{otherwise.}
\end{cases}
\]
\end{thm}

\begin{thm} \label{furmanthm} \emph{(Furma{\'n}czyk, Kowsalya, Vernold Vivin)}
Let $C_n$ be a cycle on $n \geq 3$ vertices. Then
\[
\chi_S(S(C_n))
\begin{cases}
=4 & \textrm{if} \ n \not\equiv 1 \mod 3 \ \textrm{and} \ n \neq 5 \\
\leq 5 & \textrm{otherwise.}
\end{cases}
\]
\end{thm}

To resolve the case of splitting graphs of cycles $C_n$ with $n \equiv 1 \pmod 3$, we first present a construction that shows that $\chi_S(S(C_n))$ is 4 for $n \geq 10$.

\begin{thm}\label{construction}
If $n \equiv 1 \mod 3$ and $n \geq 10$, then $\chi_S(S(C_n)) = 4$.
\end{thm}

\begin{proof}
Let $n \in \N$ with $n \equiv 1 \pmod 3$ and $n \geq 10$. By \cite{furman}, we know that for all $n$, $\chi_S(S(C_n)) \geq 4$. We now provide a construction to star color $S(C_n)$ with four colors. Label the vertices of the copy of $C_n$ in $S(C_n)$ with $v_0, v_1, \ldots , v_{n-1}$ clockwise. Label the vertex corresponding to $v_i$ in the splitting graph construction with $v_i'$ for $0 \leq i \leq n-1$.

Define $\phi : V(S(C_n)) \to \{1,2,3,4\}$ as follows. For $0 \leq i \leq n-8$:
\[
\phi(v_i) =
\begin{cases}
1 & \textrm{if} \ n \equiv 0 \pmod 3;\\
2 & \textrm{if} \ n \equiv 1 \pmod 3;\\
3 & \textrm{if} \ n \equiv 2 \pmod 3.
\end{cases}
\]

We color the remaining seven vertices of $C_n$ as follows. Let $\phi(v_{n-1})=\phi(v_{n-4})=4$, $\phi(v_{n-2})=\phi(v_{n-5})= 3$, $\phi(v_{n-3})=\phi(v_{n-7})=1$, and $\phi(v_{n-6})=2$.

Now, we color the splitting vertices. Let $\phi(v_i')=4$ for $1 \leq i \leq n-6$. Let $\phi(v_j')=2$ for $n-4 \leq j \leq n-1$. Let $\phi(v_0')=3$ and $\phi(v_{n-5}')=1$. 

We claim that $\phi$ is a star coloring of $S(C_n)$. The proof follows from Figure \ref{Fig1} and Figure \ref{Fig2}. Figure \ref{Fig1} shows the four star coloring of $S(C_{10})$. If $n=10+3k$, then our construction for $\phi$ essentially glues a copy of Figure \ref{Fig2} with $3k$ nodes at the appropriate spot (marked by dotted lines) in Figure \ref{Fig1}. It is easy to check that this does not create any new 2-colored $P_4$'s and so $\phi$ is a four star coloring of $S(C_n)$.
\end{proof}

\begin{figure}[t]
\begin{center}
\includegraphics[height=3cm]{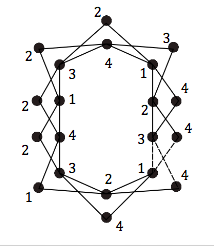}
\end{center}
\caption{The four star coloring of $S(C_{10})$ constructed in Theorem \ref{construction}.} \label{Fig1}
\end{figure}

\begin{figure}[t]
\begin{center}
\includegraphics[height=3cm]{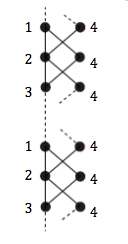}
\end{center}
\caption{The four star coloring of the "insertion" piece from Theorem \ref{construction}.} \label{Fig2}
\end{figure}

We will now argue that the star-chromatic numbers of $S(C_4)$, $S(C_5)$, and $S(C_7)$ are five. All three proofs have essentially the same flavor with some additional technical detail for $S(C_7)$. The idea of two colored graphs being ``the same" will be useful and so we provide a formal definition:

\begin{definition}
Let $G_1$ and $G_2$ be two graphs with vertex sets $V(G_1)$ and $V(G_2)$ respectively and vertex colorings $\phi_1$ and $\phi_2$ respectively. Then, $G_1$ and $G_2$ are \emph{isomorphic as vertex-colored graphs} if there are bijective functions $\pi: V(G_1) \to V(G_2)$ and $\theta: \{\phi_1(v) \ : \ v \in V(G_1)\} \to \{\phi_2(v) \ : \ v \in V(G_2)\}$ such that $u$ is adjacent to $v$ in $G_1$ if and only if $\pi(u)$ is adjacent to $\pi(v)$ in $G_2$ and for all $v \in V(G_1)$ we have $\theta(\phi_1(v))=\phi_2(\pi(v))$.
\end{definition}

\begin{thm}
The star chromatic number of the splitting graph of $C_4$ is 5.
\end{thm}

\begin{proof}
Label the vertices of $C_5$ clockwise with $v_0,v_1,v_2,v_3$ and the vertex corresponding to $v_i$ under the splitting graph construction with $v_i'$.
Suppose for sake of contradiction that $\phi$ is a four star-coloring of $S(C_4)$. By Theorem \ref{cycle}, it suffices to consider the following two cases.

\textbf{Case 1}: Suppose that $\phi$ uses all four colors to color the copy of $C_4$ in $S(C_4)$. Since $\phi$ is a proper vertex coloring, either $\phi(v_0') = \phi(v_0)$ or $\phi(v_0')=\phi(v_2)$. If $\phi(v_0')=\phi(v_0)$, then it follows that $\phi(v_1')=\phi(v_3)$. Then, $v_1' \to v_0 \to v_3 \to v_0'$ is a 2-colored $P_4$, a contradiction. On the other hand, if $\phi(v_0')=\phi(v_2)$, then it follows from considering the path $v_0' \to v_1 \to v_2 \to v_3'$ that $\phi(v_3')=\phi(v_3)$. Now, $v_0' \to v_3 \to v_2 \to v_3'$ is a 2-colored $P_4$, a contradiction.

\textbf{Case 2}: Suppose $\phi$ uses three colors to color the copy of $C_4$ in $S(C_4)$. Without loss of generality, we can assume $\phi(v_0)=\phi(v_2)$ (the other cases are isomorphic as vertex-colored graphs). Since $\phi$ is a proper vertex coloring, either $\phi(v_3') = \phi(v_3)$ or $\phi(v_3')=\phi(v_1)$. In the former case, $v_3' \to v_2 \to v_3 \to v_0$ is a 2-colored $P_4$ and in the latter case, $v_3' \to v_0 \to v_1 \to v_2$ is a 2-colored $P_4$, a contradiction.
\end{proof}

\begin{thm}
The star chromatic number of the splitting graph of $C_5$ is 5.
\end{thm}

\begin{proof}
Label the vertices of $C_5$ clockwise with $v_0, \ldots, v_4$ and label the vertex corresponding to $v_i$ under the splitting graph construction with $v_i'$.

Suppose for the sake of contradiction that $\phi$ is a four star coloring of $S(C_5))$. By Theorem \ref{cycle}, $\phi$ uses all four colors to color $C_5$. Since $\phi$ is a proper star coloring, $\phi$ must use one of the four colors to color two distinct vertices and the other three colors to color the remaining three vertices. Without loss of generality, we can assume $\phi(v_0)$ is used twice. We now consider two cases depending on which other vertex of $C_5$ has the same color as $v_0$.

\textbf{Case 1}: Suppose $\phi(v_0) = \phi(v_2)$. This implies $\phi(v_4) = \phi(v_4')$ and therefore that $\phi(v_0')=\phi(v_3)$. Then, $v_0' \to v_4 \to v_3 \to v_4'$ is a 2-colored $P_4$, a contradiction.

\textbf{Case 2}: Suppose $\phi(v_0) = \phi(v_3)$. This implies $\phi(v_1)=\phi(v_1')$ and therefore that $\phi(v_0') = \phi(v_2)$. Then, $v_0' \to v_1 \to v_2 \to v_1'$ is a 2-colored $P_4$, a contradiction.

Thus, $\chi_S(S(C_5)) = 5$.
\end{proof}

To prove that the star chromatic number of $S(C_7)$ is 5, we first give two helpful lemmas.

\begin{lem}\label{bicolp3}
Suppose $C_7$ is three star colored. Then, some $P_3$ in $C_7$ is 2-colored.
\end{lem}

\begin{proof}
Suppose for the sake of contradiction that $\phi: V(C_7) \to \{1,2,3\}$ is a three star coloring of $C_7$ with no bi-colored $P_3$. Label the vertices of $C_7$ clockwise with $v_0, \ldots , v_6$. Since $\phi$ is a star coloring, for some $i$ we know $\phi(v_i)=1$. Since $\phi$ has no bi-colored $P_3$, we know $\phi(v_{i-1}) \neq \phi(v_{i+1})$. Suppose without loss of generality (the other case is symmetric) that $\phi(v_{i-1})=2$ and $\phi(v_{i+1})=3$. The fact that $\phi$ has no 2-colored $P_3$ completely determines the colors of the remaining vertices and it follows that $v_{i-2} \to v_{i-3} \to v_{i+3}$ is a 2-colored $P_3$, a contradiction.
\end{proof}

\begin{lem}\label{babcb}
Label the vertices of $C_n$ clockwise with $v_0, \ldots, v_{n-1}$ and label the vertex corresponding to $v_i$ under the splitting graph construction with $v_i'$. Suppose $\phi$ is a $k$-star coloring of $S(C_n)$ and there exists $i$ such that $\phi(v_i)=\phi(v_{i+2})=\phi(v_{i+4})$. Then $k \geq 5$.
\end{lem}

\begin{proof}
Since $\phi$ is a star coloring, $\phi(v_{i+1}) \neq \phi(v_{i+3})$. It follows that $\phi$ assigns $v_{i+1}'$ a different color from the three distinct colors used to color $v_i$, $v_{i+1}$, and $v_{i+3}$ (if the color of $v_{i+1}$ matches any of the three colors used to color $v_i$, $v_{i+1}$, or $v_{i+3}$, then it is easy to check that $\phi$ is not a star coloring). Considering the path $v_i \to v_{i+1}' \to v_{i+2} \to v_{i+3}'$, we see that $\phi(v_{i+1}') \neq \phi(v_{i+3}')$. Thus, $k \geq 5$.
\end{proof}

\begin{figure}[t]
\begin{center}
\includegraphics[height=2cm]{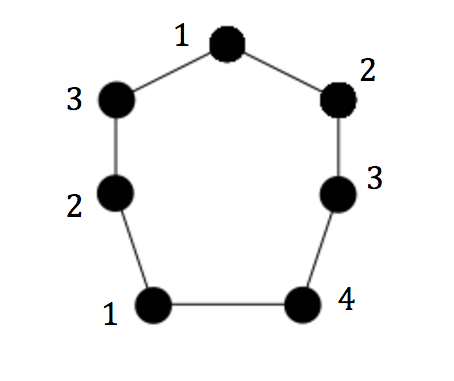}
\end{center}
\caption{The four star coloring of $C_7$ used in Case 2b of the proof of Theorem \ref{c7}} \label{Fig3}
\end{figure}

\begin{thm} \label{c7}
The star chromatic number of $S(C_7)$ is 5.
\end{thm}

\begin{proof}
Suppose for the sake of contradiction that $\phi$ is a four star-coloring of $S(C_7)$. We then consider two cases depending on whether $\phi$ uses three or four colors to color $C_7$.

\textbf{Case 1}: Suppose $\phi$ uses three colors to color $C_7$. We know by Lemma \ref{bicolp3} that, colored by $\phi$, $C_7$ has some 2-colored $P_3$. That is, for some $i$, $\phi(v_{i-1})=\phi(v_{i+1})$. Since $\phi$ is a proper vertex coloring, $\phi(v_{i+2}) \neq v_{i-1})$ and similarly, $\phi(v_{i-2}) \neq \phi(v_{i-1})$. Since $\phi$ is a star coloring, we have that $\phi(v_{i_2}) \neq \phi(v_i)$ and $\phi(v_{i+2}) \neq \phi(v_i)$. This implies that either $\phi(v_{i-3})=\phi(v_{i+1})$ or $\phi(v_{i+3})=\phi(v_{i+1})$. In either case, Lemma \ref{babcb} gives a contradiction.

\textbf{Case 2}: Suppose $\phi$ uses four colors to color $C_7$. It is easy to check that $\phi$ cannot assign any one color to three distinct vertices of $C_7$. The proof now proceeds through two subcases.

\textbf{Case 2a}: Suppose $\phi$ has a bi-colored $P_3$. That is, there exists $i$ such that $\phi(v_{i-1})=\phi(v_{i+1})$. Since $\phi$ is a star-coloring, $\phi(v_{i-2}) \neq \phi(v_i)$ and $\phi(v_{i+2}) \neq \phi(v_i)$. It is easy to see that $\phi(v_{i-2}) = \phi (v_{i+2})$ (otherwise, $v_i'$ requires a fifth color to be properly star colored). Since $\phi$ is a star coloring that uses four colors to star color $C_7$, one of $\phi(v_{i-3})$ or $\phi(v_{i+3})$ is distinct from $\phi(v_i)$, $\phi(v_{i+1})$, and $\phi(v_{i+2})$. Because the two cases are symmetric, we may suppose $\phi(v_{i-3})$ is the distinct color. Lemma \ref{babcb} and the fact that $\phi$ is a proper vertex coloring imply that $\phi(v_{i+3})=\phi(v_i)$. Now, it follows that $\phi(v_i') = \phi(v_{i-3})$. This implies $\phi(v_{i+2}') = \phi(v_{i+2})$ and thus $\phi(v_{i+3}') = \phi(v_{i+1})$. Then, $v_{i+3}' \to v_{i+2} \to v_{i+1} \to v_{i+2}'$ is a 2-colored $P_3$, a contradiction.

\textbf{Case 2b}: Assume next that $\phi$ has no bi-colored $P_3$. One can check that any four star coloring of $C_7$ without any 2-colored $P_3$ is isomorphic as a vertex coloring to the coloring presented in Figure \ref{Fig3}.

First, we claim that for some $i$, $\phi(v_i) = \phi(v_i')$. Suppose for the sake of contradiction that for all $i$, $\phi(v_i) \neq \phi(v_i')$. Because $\phi$ has no 2-colored path with three vertices in $C_7$, this fully determines the color of each of the vertices added in the splitting construction. If we label the vertices of $C_7$ (clockwise by $v_0, \dots , v_6$) such that $\phi(v_i)=4$, then we can check that this leads to the 2-colored $P_4$ $v_{i-2}' \to v_{i-1} \to v_i \to v_{i+1}'$, a contradiction.

Thus, there exists an $i$, $0 \leq i \leq 6$, such that $\phi(v_i)=\phi(v_i')$. Since $C_7$ has no 2-colored path with three vertices, the condition $\phi(v_i)=\phi(v_i')$ for any $i$ fully determines $\phi$. It is easy to check that for a fixed $i$ with $0 \leq i \leq 6$, the coloring $\phi$ determined by setting $\phi(v_i)=\phi(v_i')$ contains a 2-colored $P_3$.
\end{proof}

\section{Conclusion}
Here we have computed the star chromatic number of splitting graphs of cycles. It would be interesting to consider star colorings of splitting graphs of other families--including complete multipartite graphs and direct products (alternatively called tensor products or Kronecker products) of cycles and paths. The shadow graph is another common construction in graph labeling (see \cite{gallian}). One could also bound the star chromatic number of the shadow graphs of various basic families.  

\section{Acknowledgments}
This research was conducted at the 2017 REU at the University of Minnesota Duluth, supported by  NSF/DMS-1659047. We would like to thank Joe Gallian for his incredible support at the REU as well as for reading through this paper.

\end{document}